\documentclass[12pt,reqno]{amsart}
\usepackage[headings]{fullpage}
\usepackage{amssymb,amsmath,amscd,bbm,tikz}
\usepackage[all,cmtip]{xy}
\usepackage{url}
\usepackage[bookmarks=true,%
    colorlinks=true,%
    linkcolor=blue,%
    citecolor=blue,%
    filecolor=blue,%
    menucolor=blue,%
    urlcolor=blue,%
    breaklinks=true]{hyperref}
\usepackage{slashed}    
\usepackage{verbatim}

\newtheorem{theorem}{Theorem}[section]
\theoremstyle{definition}
\newtheorem{proposition}[theorem]{Proposition}
\newtheorem{lemma}[theorem]{Lemma}

\newtheorem{remark}[theorem]{Remark}

\def\BN{\mathbbm N}
\def\BZ{\mathbbm Z}
\def\BQ{\mathbbm Q}

\def\BT{\mathbbm T}
\def\BE{\mathbbm E}

\def\calT{\mathcal T}

\def\calS{\mathcal S}

\def\ID{I_{\Delta}}
\def\JD{J_{\Delta}}

\def\d{\delta}

\def\be{\begin{equation}}
\def\ee{\end{equation}}

\def\IFKB{I^{\mathrm{FKB}}}
\def\TV{\mathrm{TV}}
\def\Ss{\Sigma}
\def\FKB{\mathrm{FKB}}
\newcommand{\qbinom}[2]{\genfrac{[}{]}{0pt}{}{#1}{#2}}
\def\U{\mathrm{U}}
\def\Tet{\mathrm{Tet}}
\def\lt{\mathrm{lt}}
\def\wh{\widehat}
\def\ev{\mathrm{ev}}

\begin{document}
\title[The FKB invariant is the 3d index]{The FKB invariant is the 3d index}
\author{Stavros Garoufalidis}
\address{
  International Center for Mathematics, Department of Mathematics \\
  Southern University of Science and Technology \\
  Shenzhen, China \newline
  {\tt \url{http://people.mpim-bonn.mpg.de/stavros}}}
\email{stavros@mpim-bonn.mpg.de}
\author{Roland van der Veen}
\address{Bernoulli Institute \\
  University of Groningen \\
  P.O. Box 407, 9700 AK Groningen \\
  The Netherlands \newline
         {\tt \url{http://www.rolandvdv.nl}}}
\email{r.i.van.der.veen@rug.nl}

\date{28 February 2020}

\begin{abstract}
  We identify the $q$-series associated to an 1-efficient ideal triangulation
  of a cusped hyperbolic 3-manifold by Frohman and Kania-Bartoszynska with
  the 3D-index of Dimofte-Gaiotto-Gukov. This implies the topological
  invariance of the $q$-series of Frohman and Kania-Bartoszynska for
  cusped hyperbolic 3-manifolds. Conversely, we identify the
  tetrahedron index of Dimofte-Gaiotto-Gukov as a limit of quantum 6j-symbols. 
\end{abstract}

\maketitle

{\footnotesize
\tableofcontents
}


\section{Introduction}
\label{sec.intro}

In their seminal paper, Turaev-Viro~\cite{TV} defined topological
invariants of triangulated 3-manifolds using state sums whose building block
are the quantum 6j-symbols at roots of unity. An extension of the Turaev-Viro
invariants to ideally triangulated 3-manifolds was given by
Turaev~\cite{Turaev:shadow1,Turaev:shadow1} and
Benedetti-Petronio~\cite{BP:roberts}.

In~\cite{FK} Frohman and Kania-Bartoszynska (abbreviated by FKB)
aimed to construct topological
invariants of ideally triangulated 3-manifolds away from roots of unity, and
with this goal in mind, they studied some limits of quantum 6j-symbols
and associated analytic functions to suitable ideal triangulations. Their
results apply to compact, oriented 3-manifolds with arbitrary boundary, but
for simplicity, throughout our paper, we will assume that $M$ is a compact,
oriented 3-manifold with torus boundary components. In that case, FKB
assigned to an 1-efficient ideal triangulation $\calT$
of such a 3-manifold $M$ a formal power series $\IFKB_\calT(q) \in \BZ[[q]]$
which turns out to be analytic in the open unit disk $|q|<1$ and which is a
generating series of suitable closed oriented surfaces carried by the
spine associated to $\calT$. FKB did not prove
that their building block satisfies the 2--3 Pachner moves of 1-efficient
triangulations, although this, together with the conjectured topological
invariance, is implicit in their work. 

In a different direction, in~\cite{DGG1,DGG2} Dimofte-Gaiotto-Gukov
(abbreviated DGG) studied
the index of a superconformal $N=2$ gauge theory via a 3d-3d correspondence.
Using as a building block an explicit formula for the partition function
$\ID$ of an ideal tetrahedron, they associated an invariant
$I_\calT(m,e)(q) \in \BZ[[q^{1/2}]]$ to a suitable ideal triangulation $\calT$
of a 3-manifold $M$ where the labels $(m,e)$ are elements of
$H_1(\partial M, \BZ)$. The construction of DGG is
predicted by physics to be a topological invariant, and indeed DGG proved
that their invariant is unchanged under suitable 2--3 Pachner moves. 

It turns out that the ideal triangulations with a well-defined 3D-index
are exactly those that satisfy a combinatorial PL condition known as an
index structure (see~\cite[Sec.2.1]{Ga:index}), and, equivalently,
those that are
1-efficient (see~\cite{GHRS}). Moreover, in~\cite{GHRS}, it was shown that the
3D-index of an 1-efficient triangulation gives rise to an invariant of a
cusped hyperbolic 3-manifold $M$ (with nonempty boundary). 

Thus, 1-efficient ideal triangulations is a common feature of the work of
FKB and DGG. A second common feature is the presence of
(generalized) normal surfaces.
On the one hand, the FKB invariant is a generating series of suitable
surfaces carried by the spine of an ideal triangulation $\calT$. 
On the other hand, it was shown in~\cite{Ga:normal} that the 3D-index
can be written as the generating series of generalized spun normal surfaces,
(these are surfaces that intersect each ideal tetrahedron in polygonal disks)
where the latter are encoded by their quadrilateral coordinates.

Given these coincidences, it is not surprising that the invariants
of ideal triangulations of~\cite{FK} and~\cite{DGG1,DGG2} coincide.

\begin{theorem}
  \label{thm.1}
  If $\calT$ is an 1-efficient triangulation, then for all elements $(m,e)
  \in H_1(\partial M,\BZ)$ we have:
  \be
  \label{eq.thm1}
  \IFKB_\calT(m,e)(q) = I_\calT(m,e)(q) \,.
  \ee
  \end{theorem}

It follows that $\IFKB$ is a topological invariant of cusped hyperbolic
3-manifolds.

Theorem~\ref{thm.1} follows from the fact that both invariants can be
expressed as generating series of surfaces whose with local weights and
the weights match (see Proposition~\ref{prop.1} below).
Recall that the tetrahedron index is given by~\cite{DGG1} 
\be
\label{eq.ID}
  I_\Delta(m,e)(q) =
  \sum_{n}(-1)^n\frac{q^{\frac{1}{2}n(n+1)-(n+\frac{1}{2}e)m}}{(q)_n(q)_{n+e}} 
\ee
where, for a natural number $n$, we define $(q;q)_n = \prod_{j=1}^n (1-q^j)$
and the summation in~\eqref{eq.ID} is over the integers $n \geq \max\{0,-e\}$. 
Although the tetrahedron index is a function of a pair of integers, 
it can be presented as a function of three
variables $a,b,c \in \BZ$~\cite[Eqn.(8)]{Ga:normal} by
\be
\label{Isym}
\JD(a,b,c)= (-q^{\frac{1}{2}})^{-b} \ID(b-c,a-b)
= (-q^{\frac{1}{2}})^{-c} \ID(c-a,b-c) = (-q^{\frac{1}{2}})^{-a} \ID(a-b,c-a).
\ee
Then $\JD(a,b,c)$ is invariant under all permutations of its arguments $a,b,c$
and satisfies the translation property
\be
\label{It}
\JD(a,b,c) = (-q^{\frac{1}{2}})^{s} \JD(a+s,b+s,c+s)  
\text{ for all } s \in \BZ \,.
\ee
The leading term of $\JD(a,b,c)$ is given by
$(-q^{\frac{1}{2}})^{\nu(a,b,c)}$ (see Eqn.(8) of~\cite{Ga:normal}) where
\be
\label{eq.nu}
\nu(a,b,c)=
a^ *b^* + a^* c^* + b^* c^* - \min\{a,b,c\}
\ee
where $a^*=a-\min\{a,b,c\}$, $b^*=b-\min\{a,b,c\}$ and
$c^*=c-\min\{a,b,c\}$.

Consider the function
\be
\label{eq.JFK}
\JD^{\FKB}(a,b,c) =
(q)_\infty  \sum_{n} (-1)^n 
\frac{q^{\frac{1}{2}n(3n+1) +n(a+b+c)+\frac{1}{2}(ab+bc+ca)}}{
  (q;q)_{n+a}(q;q)_{n+b}(q;q)_{n+c}}
\ee
for integers $a,b$ and $c$, where the summation is over the integers
(with the understanding that $(q;q)_m=\infty$ when $m<0$), or alternatively
over the integers $n \geq -\min\{a,b,c\}$.
FKB identify the above function 
as a limit of quantum 6j-symbols. It turns out that the
limit is equivalent to the stabilization of the coefficients of the quantum
6j-symbols, and the latter follows from degree estimates.
To state our result, consider the building blocks $\Theta$ and $\Tet$
(functions of three and six integer variables, respectively)
whose definition is given explicitly in Equations~\eqref{eq.b4} and
~\eqref{eq.b5} of Section~\ref{sub.blocks}. Denote by $\wh \Theta$ and
$\wh \Tet$ the shifted versions defined in Section~\ref{sub.stab}. Then we
have the following.

\begin{proposition}
  \label{prop.stab}
  We have: 
  \begin{align}
    \label{eq.3jlim}
    \lim_{N\to \infty} \wh \Theta (a+2N,b+2N,c+2N) &=
    \frac{1}{1-q} \frac{1}{(q;q)_\infty^2} \\
    \label{eq.6jlim}
    \lim_{N\to \infty} \wh \Tet \begin{pmatrix} a +2N & b+2N & e+2N
      \\ d +2N & c + 2N & f + 2N \end{pmatrix}&=
(-q^{-\frac{1}{2}})^{\nu(S_1^*,S_2^*,S_3^*)}                              
\frac{1}{(1-q)(q;q)_\infty^4} \JD^\FKB(S_1^*,S_2^*,S_3^*) 
  \end{align}
  where $S_i$ are given in~\eqref{eq.Sj}, $S^*=\min\{S_1,S_2,S_3\}$
  and $S_i^* = S_i-S^*$.  
\end{proposition}

Observe that the quantum $6j$-symbols depend on six parameters (one
per edge of the tetrahedron) while its limit given by~\eqref{eq.6jlim}
depends only on three parameters (one for each quadrilateral of the
tetrahedron), and a further symmetry reduces the dependence to two
parameters (obtained by ignoring one of the three quadrilateral types
of the tetrahedron). 

The next proposition identifies the tetrahedron index of~\cite{DGG1,DGG2}
as a limit of quantum 6j-symbols. 

\begin{proposition}
\label{prop.1}
For integers $a,b$ and $c$ we have: 
\be
\label{eq.prop1}
\JD^{\FKB}(a,b,c) =
\JD(a,b,c) \,.
\ee
\end{proposition}


\section{A review of~\cite{TV} and~\cite{FK}}
\label{sec.FK}

\subsection{The building blocks}
\label{sub.blocks}

In this section we review the construction of the Turaev-Viro invariant
and the results of~\cite{FK}. Those invariants use some building blocks
whose definition we recall now. Note that the normalization of the building
blocks is not standard in the literature, and we 
will use the standard definitions of the building blocks that can
be found in ~\cite{KL} and also in~\cite{MV}. 
Recall the {\em quantum integer} $[n]$ and the {\em quantum factorial}
$[n]!$ of a natural number $n$ are defined by
$$
[n]=\frac{q^{n/2}-q^{-n/2}}{q^{1/2}-q^{-1/2}},
\qquad
[n]!=\prod_{k=1}^n [k]!
$$
with the convention that $[0]!=1$. Let 
$$
\qbinom{a}{a_1, a_2, \dots, a_r}=\frac{[a]!}{[a_1]! \dots [a_r]!}
$$
denote the multinomial coefficient of natural numbers $a_i$ such that
$a_1+ \dots + a_r=a$. We say that a triple $(a,b,c)$ of natural
numbers is {\em admissible} if $a+b+c$ is even and the triangle inequalities
hold. In the formulas below, we use the following basic trivalent graphs
$\U,\Theta,\Tet$ colored by one, three and six natural numbers (one
in each edge of the corresponding graph) such that
the colors at every vertex form an admissible triple shown in 
Figure~\ref{f.3j6j}.

\begin{figure}[!htpb]
\begin{center}
\includegraphics[height=0.10\textheight]{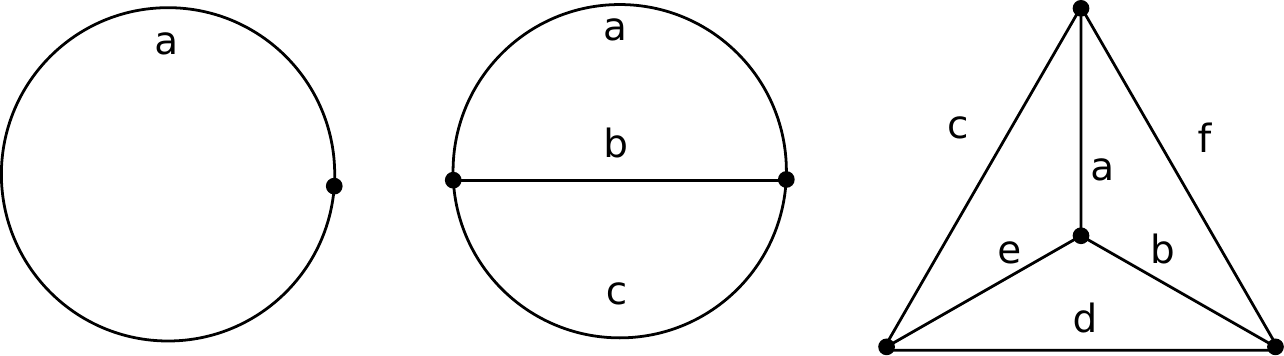}
\end{center}
\caption{The Unknot, the $\Theta$ graph and the tetrahedron.}
\label{f.3j6j}
\end{figure}

Let us define the following functions.


\begin{subequations}
\begin{align}
  \label{eq.b3}
  \U(a)&=
  (-1)^a[a+1] \\
  \label{eq.b4}
  \Theta(a,b,c)&=
  (-1)^{\frac{a+b+c}{2}}[\frac{a+b+c}{2}+1]
\qbinom{\frac{a+b+c}{2}}{\frac{-a+b+c}{2}, \frac{a-b+c}{2}, \frac{a+b-c}{2}}
\\
  \label{eq.b5}
  \Tet\begin{pmatrix} a & b & e \\ d & c & f \end{pmatrix}&=
  \sum_{k = T^+}^{S^*} (-1)^k [k+1]
        \qbinom{k}{S_1-k , S_2-k , S_3-k , k- T_1 , k- T_2 , k- T_3 , k- T_4}
\end{align}
\end{subequations}
where
\begin{equation}
\label{eq.Sj}
S_1 = \frac{1}{2}(a+d+b+c)\qquad S_2 = \frac{1}{2}(a+d+e+f) 
\qquad S_3 = \frac{1}{2}(b+c+e+f)
\end{equation}
\begin{equation}
\label{eq.Ti}
T_1 = \frac{1}{2}(a+b+e) \qquad T_2 = \frac{1}{2}(a+c+f)
\qquad T_3 = \frac{1}{2}(c+d+e) \qquad T_4 = \frac{1}{2}(b+d+f) 
\end{equation}
and
\be
\label{ST}
S^* = \min\{S_1,S_2,S_3 \}, \qquad T^+ = \max\{T_1,T_2,T_3,T_4 \} \,.
\ee

\subsection{The Turaev-Viro invariant}
\label{sub.TV}

Suppose $M$ is a
3-manifold as in our introduction, $\calT$ is an ideal triangulation of
$M$ and $X$ is the corresponding simple spine of $\calT$, i.e., the dual
2-skeleton of $\calT$. Let $V(X)$, $E(X)$ and $F(X)$ denote the vertices,
edges and faces of $X$. 

A admissible coloring $c: F(X)\to \BN$ of $X$ is an assignment of
natural numbers at each face of $X$ such that at each edge of $X$ the
sum of the three colors are even, and they satisfy the triangle inequality.
An admissible coloring $c$ determines a 6-tuple $(a_v,b_v,c_c,d_v,e_v,f_v)$
of integers at each vertex $v$ of $X$, a 3-tuple $(a_e,b_e,c_e)$ of integers
at each edge $e$ of $X$ and an integer $u_f$ at each face $f$ of $X$.

If $r$ is a natural number, a coloring $c$ is $r$-admissible if the sum
of the colors at each edge is $\leq 2(r-2)$. Let $\zeta_{r}$ denote a
primitive $r$th root of unity. Turaev-Viro~\cite{TV} define an invariant
\be
\label{eq.TVr}
\TV_X(\zeta_{r}) = \ev_{\zeta_{r}} \sum_{c} \prod_{v \in V(X)}
\Tet\begin{pmatrix} a_v & b_v & e_v \\ d_v & c_v & f_v \end{pmatrix}
\prod_{e \in E(X)} \Theta(a_e,b_e,c_e)^{-1} \prod_{f \in F(X)} \U(u_f)
\ee
where $\ev_{\zeta_{r}}$ denotes the evaluation of a rational function of $q$
at $q=\zeta_{r}$, and the sum is over the set of $r$-admissible colorings 
Turaev-Viro prove that the above state-sum is a topological invariant of
$M$, i.e., independent of the ideal triangulation $\calT$. An extension
of the above invariant $\TV_{(X,\gamma)}(\zeta_{r})$ can be defined by
fixing an element $\gamma \in H_1(M,\BZ)$, which determines
a spine $X(\gamma)$ (called an augmented spine in~\cite[Sec.2.2]{FK}).

\subsection{The FKB invariant}
\label{sub.FKB}
  
In~\cite{FK} it was observed that an admissible coloring $c$ of $X$ gives
rise to a surface $\Ss(c)$ of $M$ carried by $X$. These surfaces which follow
the spine and resolve the singularities were called
spinal surfaces in~\cite{FK} and they are carried by the branched surface
$X$. Spinal surfaces can be encoded by their weight coordinates, as is
natural in normal surface theory, and their Haken sum can be defined in such a
way that the sum of their weights is the weight of their sum. Thus, the weight
coordinates of spinal surfaces generate a monoid $\calS(X)$. There is a natural
increasing filtration on $\calS(X)$ where $\calS(X)_N$ denotes the (finite set
of) surfaces with maximum weight at each face at most $N$. The idea
of~\cite{FK} is to use the same building blocks where now $q$ is a complex
number inside the unit disk, and consider the sum
\be
\label{eq.TVNq}
\TV^{(N)}_X(q) = \sum_{\Sigma \in \calS(X)_N} \prod_{v \in V(F)}
\Tet\begin{pmatrix} a_v & b_v & e_v \\ d_v & c_v & f_v \end{pmatrix}
\prod_{e \in E(X)} \Theta(a_e,b_e,c_e)^{-1} \prod_{f \in F(X)} \U(u_f)
\ee
Alas, $\TV^{(N)}_X(q)$ is not a topological invariant (see below).
However, the following is true.

\begin{theorem}
  \cite{FK}
  \label{thm.FK1}
  Fix a 1-efficient ideal triangulation $\calT$ of a 3-manifold $M$ with
  torus boundary components and let $X$ be the dual spine. Then, the
  following limit exists
    \be
    \label{TVlim}
    \IFKB_\calT(q) := \lim_{N\to \infty} \frac{2}{N} \TV^{(N)}_X(q)
    \in \BZ[[q]] \,. 
    \ee
  \end{theorem}

\begin{remark}
  \label{rem.2N}
  The limit in~\eqref{TVlim} is a correction of~\cite[Thm.5.1(ii)]{FK}
  where with the notation of~\cite{FK}, one has $k=0,\dots,N/2$.
\end{remark}

The existence of the above limit is only the beginning of a stability
of the coefficients of the sequence $\TV^{(N)}_X(q)$ in the sense of
asymptotic expansions of sequences in the Laurent polynomial ring
$\BZ((q^{\frac{1}{2}}))$ discussed in~\cite{GL:tails}. In examples,
it appears that the sequence $\TV^{(N)}_X(q)$ stabilizes to a quasi-linear
function, i.e., that we have:
\be
\label{eq.TVstab}
  \lim_{N} \TV^{(N)}_X(q) - \frac{N}{2} \IFKB_\calT(q) = \IFKB_{(0),\calT}(q)
  + \IFKB_{(1),\calT}(q) \cdot
  \begin{cases}
    0 & N \,\, \text{even} \\
    1 & N \,\, \text{odd}
  \end{cases}
\ee
where $\IFKB_\calT(q), \IFKB_{(0),\calT}(q)$ and $2 \IFKB_{(1),\calT}(q)
\in \BZ[[q^{1/2}]]$. However, $\IFKB_{(0),\calT}(q)$ and $\IFKB_{(1),\calT}(q)$
depend on the triangulation. For example, for the standard ideal triangulation
of the figure eight knot complement $\calT_{4_1,2}$ with two tetrahedra
(and isometry signature \texttt{cPcbbbiht}) we have
\begin{align*}
\IFKB_{\calT_{4_1,2}}(q) &= 1-2q-3q^2+2q^3+8q^4+18q^5+\dots
\\
\IFKB_{(0),\calT_{4_1,2}}(q) &= 1+4q^2+4q^3-6q^4-36q^5+\dots
\\
2 \IFKB_{(1),\calT_{4_1,2}}(q) &=
-1+2q+3q^2-2q^3-8q^4-18q^5+\dots \,,
\end{align*}
whereas for the geometric triangulation $\calT_{4_1,3}$ of the figure eight
knot complement with three tetrahedra (and isometry
signature~\texttt{dLQbcccdegj}) we have
$$
\IFKB_{\calT_{4_1,3}}(q) = \IFKB_{\calT_{4_1,2}}(q)
$$
as expected but
\begin{align*}
  \IFKB_{(0),\calT_{4_1,3}}(q) &= 1+4q^2+4q^3-6q^4-36q^5+\dots
  \\
2\IFKB_{(1),\calT_{4_1,3}}(q) &= 1+2q+2q^2+8q^3-12q^4-72q^5+\dots 
\end{align*}

The next result of~\cite{FK} identifies the above limit with a generating
series of the monoid of spinal surfaces, modulo the boundary torii. Such
surfaces were called unpeelable in~\cite{FK}. Define the weight
$E_\infty(\Ss)$ of a spinal surface $\Ss$ to be 
\be
\label{eq.Einf}
  E_\infty(\Ss) = (-q^{\frac{1}{2}})^{-\chi(\Ss)} \prod_f 
  \frac{1}{1-q}\prod_v S_\infty
  \begin{pmatrix} a_v & b_v & e_v \\ d_v & c_v & f_v \end{pmatrix}
  \ee
where if $C_1\geq C_2\geq C_3$ the sums of opposite edge weights of the
tetrahedron, $\alpha = \frac{C_1-C_3}{2}$,  $\beta = \frac{C_1-C_2}{2}$,
then
\begin{align}
\label{eq.Sinf}
  S_\infty\begin{pmatrix} a & b & e \\ d & c & f \end{pmatrix}
                                               &=
  (1-q)(q)_\infty\sum_{n=0}^\infty (-1)^{n}
  \frac{q^{\frac{3}{2}n^2+(\alpha+\beta+\frac{1}{2})n+\frac{1}{2}\alpha \beta}}{
    (q)_n(q)_{n+\alpha}(q)_{n+\beta}}
  \\ &=
  (1-q)(q)_\infty \JD^\FKB(S^*_1,S^*_2,S^*_3)
\end{align}
where $S_1 \geq S_2 \geq S_3$ thus $S_3^*=0$ and 
$\alpha = S_1-S_3= \frac{C_1-C_3}{2}$ and $\beta = S_2-S_3 =\frac{C_1-C_2}{2}$.
It follows that for a spinal surface $\Ss$ we have
\be
\label{eq.Sinf2}
E_\infty(\Ss) = (-q^{\frac{1}{2}})^{-\chi(\Ss)} \JD^\FKB(\Ss) 
\ee
where
\be
\label{eq.Sinf3}
\JD^\FKB(\Ss) = \prod_{j=1}^t
\JD^\FKB(a_j,b_j,c_j)
\ee
and $(a_1,b_1,c_1,\dots,a_t,b_t,c_t)$ are the quad coordinates of $\Ss$ and
$t$ is the number of tetrahedra of $\calT$.
Note that $\Ss$ is unpeelable if and only if $\min\{a_j,b_j,c_j\}=0$
for all $j=1,\dots,t$.

\begin{theorem}
  \cite{FK}
  \label{thm.FK2}
    Under the assumptions of Theorem~\ref{thm.FK1}, the limit coincides
    with the generating series of closed unpeelable surfaces carried
    by the spine of $\calT$
    \be
    \label{TVlim2}
    \IFKB_\calT(q) = \sum_{\Ss \, : \, \text{unpeelable}}E_\infty(\Ss) \,.
    \ee
\end{theorem}
It is possible to extend Theorems~\ref{thm.FK1} and~\ref{thm.FK2} using
an element $\gamma \in H_1(\partial M,\BZ)$.
Consider the augmented spine $X(\gamma)$. Then one can define
$\TV^{(N)}_X(\gamma)(q)$ and the corresponding limit
$\IFKB_\calT(\gamma)(q)$ exists and is identified with the generating series
of unpeelable surfaces $\Ss$ with boundary $\gamma$.  


\section{Proofs}
\label{sec.proofs}

\subsection{Stabilization of the building blocks}
\label{sub.stab}

In this section we prove some stabilization properties of the building
blocks of quantum spin networks, using elementary degree estimates, in
the spirit of~\cite{GL:tails}, where the stabilization of the coefficients
of the colored Jones polynomial of an alternating knot was proven, giving
rise to a sequence of $q$-series, the first of which is known as the
tail of the colored Jones polynomial. 

We begin by expressing the building blocks of Section~\ref{sub.blocks}
in terms of the quantum
factorial $(q;q)_n$ where $(qx;q)_n=\prod_{j=1}^n (1-q^j x)$ for $n$
a nonnegative integer. We have:
$$
[n] = q^{-\frac{n-1}{2}} \frac{1-q^n}{1-q}, \qquad
[n]!= q^{-\frac{n(n-1)}{4}} \frac{(q;q)_n}{(1-q)^n}
$$
and
$$
\qbinom{a}{a_1, a_2, \dots, a_r}=\frac{[a]!}{[a_1]! \dots [a_r]!} =
q^{-\frac{1}{4}(a^2-\sum_{j=1}^r a_j^2)}
\frac{(q;q)_a}{(q;q)_{a_1} \dots (q;q)_{a_r}} \,.
$$
and
\be
\label{6jb}
(\Tet)\begin{pmatrix} a & b & e \\ d & c & f \end{pmatrix} =
\sum_{k=T^+}^{S^*} (-1)^k \frac{1-q^k}{1-q} q^{
  \delta(\Tet)\begin{pmatrix} a & b & e \\ d & c & f \end{pmatrix}}
\frac{(q;q)_k}{\prod_{i=1}^3 (q;q)_{S_i-k} \prod_{j=1}^4 (q;q)_{k-T_j}} 
  \ee
  where $\delta(\Tet)\begin{pmatrix} a & b & e \\ d & c & f \end{pmatrix}$
  is defined in Lemma~\ref{lem.tropicalblocks} below and $S_i$ and $T_j$
  are given in Equations \eqref{eq.Sj} and \eqref{eq.Ti}.
  
Since $(q;q)_n \in \BZ[q]$ is a polynomial with constant term 1, it follows
that $1/(q;q)_n \in \BQ(q) \cap \BZ[[q]]$. The building blocks are rational
functions of $q$ with denominators products of cyclotomic polynomials, hence
they are well-defined elements of the Laurent polynomial ring $\BZ((q))$.
If $f(q) \in \BZ((q))$ we will denote by $\lt(f)q^{\delta(f)}$ the monomial
with the lowest power of $q$ appearing in the Laurent expansion of $f(q)$,
and we will denote $\wh f(q) = \lt(f)^{-1}q^{-\delta(f)} f(q)$ the shifted
series, which, when $\lt(f) = \pm 1$, is an element of $1+q\BZ[[q]]$. 

Note that our notation differs slightly from Section 2 of~\cite{GV:2fusion},
where we studied the leading terms of the building blocks with the aim
of computing the degree of the colored Jones polynomial.

The next lemma is elementary (see~\cite[Lem.2.4]{GV:2fusion}).

\begin{lemma}
\label{lem.tropicalblocks}
For all admissible colorings we have:
\begin{align*}
\lt(\U)(a)&=(-1)^a  \\
\lt(\Theta)(a,b,c)&= (-1)^{\frac{a+b+c}{2}}
\\
\lt(\Tet)\begin{pmatrix} a & b & e \\ d & c & f \end{pmatrix}&= (-1)^{ T^+}
\end{align*}
and
\begin{align*}
\d(\U)(a)&= \frac{a}{2} \\
\d(\Theta)(a,b,c)&= -\frac{1}{8}(a^2+b^2+c^2)+\frac{1}{4}(ab+ac+bc)
+\frac{1}{4}(a+b+c)
\\
  \d(\Tet)\begin{pmatrix} a & b & c \\ d & e & f \end{pmatrix}&=
\frac{1}{4} \left(-(T^+)^2 + \sum_i (S_i -T^+)^2 + \sum_j (T^+-T_j)^2\right)
- \frac{T^+}{2}
\end{align*}
where $S_j$ and $T_i$ are given in Equations \eqref{eq.Sj} and \eqref{eq.Ti}.
\end{lemma}

We have all the ingredients to give a proof of Proposition~\ref{prop.stab}.

\begin{proof}(of Proposition~\ref{prop.stab}) 
  The first identity follows from the fact that
  $$
  \wh \Theta(a,b,c) =
  \frac{1-q^{\frac{a+b+c}{2}+1}}{1-q} \frac{(q;q)_{\frac{a+b+c}{2}}}{
    (q;q)_{\frac{-a+b+c}{2}}(q;q)_{\frac{a-b+c}{2}}(q;q)_{\frac{a+b-c}{2}}}
  $$
  and the fact that
  \be
  \label{basic}
  \lim_N q^{\kappa + \lambda N}=0, \qquad 
  \lim_N (q;q)_{\kappa' + \lambda N} = (q;q)_\infty
  \ee
  for
  integers $\kappa$, $\kappa'$ and $\lambda$ with $\lambda >0$.

  For the second identity, 
  the sum over $k$ (with $T^+ \leq k \leq S^*$) in~\eqref{6jb}
  achieves the minimum $q$-degree uniquely at $k=T^+$. After changing
  variables to $k=S^*-\ell$ it follows that 
  $$
  \wh \Tet \begin{pmatrix} a & b & c
    \\ d & e & f \end{pmatrix}
  = \frac{1}{1-q} \sum_{\ell=0}^{S^*-T^+}
  (-1)^\ell (1-q^{S^*-\ell}) q^{\frac{1}{2}\ell(3\ell+1)+\ell(S_1^*+S_2^*+S_3^*)}
  \frac{(q;q)_{S^*-\ell}}{\prod_i (q;q)_{S_i^*+\ell}
    \prod_j (q;q)_{T_j^*-\ell}}
  $$
  where $S_i^*=S_i-S^*$ and $T_j^*=S^*-T_j$. It follows that for all natural
  numbers $N$, we have
\begin{multline}
  \wh \Tet \begin{pmatrix} a+2N & b+2N & c+2N
    \\ d+2N & e+2N & f+2N \end{pmatrix}
  = \\  \frac{1}{1-q} \sum_{\ell=0}^{N+S^*-T^+}
  (-1)^\ell (1-q^{4N+S^*-\ell}) q^{\frac{1}{2}\ell(3\ell+1)+\ell(S_1^*+S_2^*+S_3^*)}
  \frac{(q;q)_{4N+S^*-\ell}}{\prod_i (q;q)_{S_i^*+\ell}
    \prod_j (q;q)_{N+T_j^*-\ell}} \,.
\end{multline}
Equation~\eqref{basic} applied to each fixed $\ell$ implies that
\begin{multline}
\lim_N \wh \Tet \begin{pmatrix} a+2N & b+2N & c+2N
    \\ d+2N & e+2N & f+2N \end{pmatrix}
  = \\  \frac{1}{1-q} \sum_{\ell=0}^{\infty}
  (-1)^\ell q^{\frac{1}{2}\ell(3\ell+1)+\ell(S_1^*+S_2^*+S_3^*)}
  \frac{(q;q)_{\infty}}{\prod_i (q;q)_{S_i^*+\ell}
    \prod_j (q;q)_{\infty}} 
\end{multline}
and this concludes the proof.
\qed

\subsection{The tetrahedron index as a limit of $q$-6j-symbols}
\label{sub.tet6j}

In this section we give a proof of Proposition~\ref{prop.1}.
Observe that $\JD^\FKB(a,b,c)$ is symmetric under all permutations of
$(a,b,c)$. Moreover, we claim that it satisfies the translation
property~\eqref{It}. Indeed, using the definition of $\JD^\FKB$
as a sum over the integers~\eqref{eq.JFK}, it follows that
\begin{align*}
\JD^{\FKB}(a+s,b+s,c+s) &=
                          (q)_\infty
                          \\ & \cdot \sum_{n} (-1)^n 
                          \frac{q^{\frac{1}{2}n(3n+1) +n(a+s+b+s+c+s)
                          +\frac{1}{2}((a+s)(b+s)+(b+s)(c+s)+(c+s)(a+s))}}{
  (q;q)_{n+a+s}(q;q)_{n+b+s}(q;q)_{n+c+s}}
  \\
                        &=
                          (-q^{\frac{1}{2}})^s
                          (q)_\infty
\sum_m (-1)^m \frac{q^{\frac{1}{2}m(3m+1) +m(a+b+c)+\frac{1}{2}(ab+ac+bc)}}{
  (q;q)_{m+a}(q;q)_{m+b}(q;q)_{m+c}}
\end{align*}
where in the first equality we shifted variables to $n+s=m$.

Since both sides of~\eqref{eq.prop1} satisfy the translation
property~\eqref{It} and are symmetric in $(a,b,c)$, to prove the said
equation, it suffices to assume that $a \geq b \geq c=0$. 
We will use the following identities
  \begin{align}
    (qx;q)_\infty &=
\sum_{n=0}^\infty (-1)^n \frac{q^{\frac{n(n+1)}{2}}x^n}{(q)_n}
\\                    
\frac{1}{(q)_m(q)_n} &=
    \sum_{\substack{r,s,t\geq 0 \\ r+s=m, s+t=n}}
    \frac{q^{rt}}{(q)_r(q)_s(q)_t}
   \end{align}                    
whose proofs may be found for example, in Equations (7) and (13) of Section
D of~\cite{Za:dilogarithm}.
  We have:
  \begin{align*}
    (q)_\infty\sum_{k=0}^\infty (-1)^k
    \frac{q^{\frac{3}{2}k^2+(a+b+\frac{1}{2})k}}{
    (q)_k(q)_{k+a}(q)_{k+b}} &=
\sum_{k} (-1)^k
\frac{q^{\frac{3}{2}k^2+(a+b+\frac{1}{2})k}}{(q)_k(q)_{k+a}}
(q^{k+b+1};q)_\infty
    \\
&=    
\sum_{k,\ell} (-1)^{k+\ell}
         \frac{q^{\frac{3}{2}k^2+(a+b+\frac{1}{2})k+\frac{\ell(\ell+1)}{2}}}{(q)_k(q)_{k+a}(q)_\ell}q^{(k+b)\ell}
         \\
&=
\sum_{n}\sum_{\substack{k+a+\ell = n+a \\ n=k+\ell}} (-1)^n
    \frac{q^{k(k+a)}q^{\frac{1}{2}n^2+\frac{n}{2}+b n}}{
    (q)_k(q)_\ell(q)_{k+a}} \\
&=
\sum_n (-1)^n \frac{q^{\frac{1}{2}n^2+\frac{n}{2}+b n}}{(q)_n(q)_{n+a}}
\\
&= q^{-\frac{1}{2}a b} \ID(-b,a) \,.
  \end{align*}
  It follows that $I^\FKB(a,b,0)=\ID(-b,a)=\JD(b,a,0)=\JD(a,b,0)$, which
  concludes the proof of the proposition.
\end{proof}

\subsection{Proof of Theorem~\ref{thm.1}}
\label{sub.thm1}

Fix a 1-efficient ideal triangulation $\calT$ with spine $X$. Recall
the generalized normal surfaces of~\cite{GHRS} and~\cite[Sec.10]{Ga:normal}.
Each generalized normal surface $S$ has weight $I(S)$ given
by~\cite[Eqn.(25)]{Ga:normal}
\be
I(S) = (-q^{\frac{1}{2}})^{-\chi(\Ss)}
\prod_{j=1}^t \JD(a_j,b_j,c_j)
\ee
where $(a_1,b_1,c_1,\dots,a_t,b_t,c_t)$ are the quad coordinates of $\Ss$ and
$t$ is the number of tetrahedra of $\calT$. 

\begin{lemma}
  \label{lem.bijection}
There is a bijection between the closed generalized normal surfaces of $\calT$
and the closed unpeelable spinal surfaces of $X$. 
If $S$ is a generalized normal surface and $\Ss$ is the corresponding
unpeelable surface, then
\be
\label{eq.Imatch}
I(S) = E_\infty(\Ss) \,.
\ee
\end{lemma}

\begin{proof}
Using the notation
of~\cite[Sec.7]{Ga:normal}, the closed generalized normal surfaces of $\calT$
are given by $Q_0(\calT,\BZ)/\BT=(\BE+\BT)/\BT$ where $\BE$ and $\BT$ are
the subspaces of integer solutions to the normal surface equations generated
by the edges and the tetrahedra of $\calT$, respectively. Every element of
$Q_0(\calT,\BZ)$ is encoded by a vector $(a_1,b_1,c_1,\dots,a_t,b_t,c_t)
\in \BZ^{3t}$ of quadrilateral coordinates where $t$ is the number of
tetrahedra of $\calT$. Moreover, the tetrahedral solution to the gluing
equations corresponding to the $\ell$-th tetrahedron is the $3t$ vector of
integers with coordinates $(a_j,b_j,c_j)=\delta_{j,\ell}(1,1,1)$ where
$\delta_{j,\ell}=1$ if $j=\ell$ and $0$ otherwise. Thus, every
generalized normal surface $S \in (\BE+\BT)/\BT$ has coordinate vector
$(a_1,b_1,c_1,\dots,a_t,b_t,c_t) \in \BN^{3t}$ satisfying
$\min\{a_j,b_j,c_j\}=0$ for all $j=1,\dots,t$. And conversely, every
such vector corresponds to a unique generalized normal surface. On the
other hand, every unpeelable closed surface is uniquely described by its
quad coordinate vector $(a_1,b_1,c_1,\dots,a_t,b_t,c_t) \in \BN^{3t}$ satisfying
$\min\{a_j,b_j,c_j\}=0$ for all $j=1,\dots,t$, and all such vectors give
rise to unpeelable surfaces. This concludes the first part of the lemma.
The second part, i.e., Equation~\eqref{eq.Imatch} follows from
Equations~\eqref{eq.Sinf2}, \eqref{eq.Sinf3} and Proposition~\ref{prop.1}.
\end{proof}

When $(m,e)=0$, Theorem~\ref{thm.1} follows from Theorem~\ref{thm.FK2},
Lemma~\ref{lem.bijection} and the fact that the 3D-index is given
by~\cite[Cor.8.2]{Ga:normal}
\be
\label{eq.3Dnormal}
I_\calT(0,0)(q)=\sum_S I(S)
\ee
where the sum is over the set of generalized normal surfaces. When $\gamma
\in H_1(\partial M, \BZ)$, one uses the obvious extension of
Lemma~\ref{lem.bijection} along with the extension of Theorem~\ref{thm.FK2}
combined with~\cite[Def.8.1]{Ga:normal}. This concludes the proof of
the theorem.
\qed


\section*{Acknowledgments}

The second author wishes to thank the International Mathematics Center
at SUSTech University, Shenzhen for their hospitality. The authors wish to
thank Banff for inviting us at the conference on Modular Forms and 
Quantum Knot Invariants in March 2018 where the results were conceived. 


\bibliographystyle{hamsalpha}
\bibliography{biblio}
\end{document}